\documentclass{amsart}
\usepackage{a4wide}
\usepackage{amssymb,amsfonts,amsmath,amsthm,amscd, epsfig}

\usepackage[all]{xy}

\theoremstyle{plain}
\newtheorem{teo}{Theorem}[section]

\newtheorem{lema}[teo]{Lemma}

\newtheorem{prop}[teo]{Proposition}
\newtheorem{obs}[teo]{Remark}

\newcommand{\nd}{\noindent}
\newcommand{\vu}{\vspace{.1cm}}
\newcommand{\vd}{\vspace{.2cm}}
\newcommand{\vt}{\vspace{.3cm}}

\begin{document}

\title{Duality for groupoid (co)actions}

\author{Daiana Fl\^ores}
\address{Departamento de Matem\'atica, Universidade Federal de Santa Maria,
97105-900, Santa Maria, RS, Brazil} \email{flores@ufsm.br}
\author{Antonio Paques}
\address{Instituto de Matem\'atica, Universidade Federal do Rio Grande do Sul,
91509-900, Porto Alegre, RS, Brazil} \email{paques@mat.ufrgs.br}

\maketitle

\begin{abstract}
In this paper we present Cohen-Montgomery-type duality theorems for groupoid
(co)actions.
\end{abstract}

\

{\scriptsize{\it Key words and phrases:} groupoid action, groupoid
coaction, Cohen-Montgomery duality.}

{\scriptsize{\it Mathematics Subject Classification:} Primary 16S40, 16W50, 20L05, 22D35.}

\section{Introduction}

Groupoids are usually presented as small categories whose morphisms
are invertible. This notion is a natural extension of the notion of
a group. Notice that a group can be seen as a category with a unique
objet.

\vu

The notion of a groupoid action that we use in this paper arose from
the notion of a partial groupoid action as introduced in \cite{BP},
which is a natural extension
of the notion of a partial group action \cite{DE}. First,
partial ordered groupoid actions on sets were introduced in the
literature, as ordered premorphisms, by N. Gilbert \cite{Gil}. After,
partial ordered groupoid actions on rings were considered
by D. Bagio and the authors \cite{BFP} as a generalization
of partial group actions, as introduced by M. Dokuchaev and R. Exel
in \cite{DE}. And in \cite{BP} this notion was extended to the general
context of groupoids.

\vu

Our purpose is to present a generalization of Cohen-Montgomery duality
Theorem for group actions \cite[Theorem 3.2]{CM} (resp., group coactions
or, equivalently, group grading \cite[Theorem 3.5]{CM}) to the setting
of groupoid actions (resp., groupoid coactions or groupoid grading)
(see Theorems \ref{teo37} and \ref{teo45}).

\vu

This paper is organized as follows. In the next section we give preliminaries about
groupoids, groupoid actions, skew groupoid rings, weak bialgebras, groupoid gradings,
groupoid coactions and weak smash products, these later as introduced by S. Caenepeel
and E. De Groot in \cite{CG1}.  In that section we will be concerned only with  the results strictly
necessary to construct the appropriate environment to prove our main theorems, whose
proofs we set in the sections 3 (actions) and 4 (coactions).

\vu

We deal with the groupoid ring $KG$ and its dual $KG^*$, $K$ being a
commutative ring and $G$ a finite groupoid. The $K$-algebras $KG$ and $KG^*$
are perhaps the first examples of weak bialgebras that are not bialgebras.

Accordingly \cite{BP}, an action of a groupoid $G$ on a $K$-algebra $R$ is a
pair $\beta=(\{E_g\}_{g\in G}, \{\beta_g\}_{g\in G}),$ where for each
$g\in G$, $E_g=E_{gg^{-1}}$ is an ideal of $R$ and $\beta_g:E_{g^{-1}}\to E_g$ is an
isomorphism of rings satisfying some appropriate  conditions (see the subsection 2.3).

If the set $G_0$ of all identities of $G$ is finite then there exists
a one to one correspondence between the structures of left $KG$-module
algebra on $R$ and the actions $\beta$ of $G$ on $R$ such that
each ideal $E_e$ ($e\in G_0$) is unital and $R=\bigoplus_{e\in G_0} E_e$.
In particular, the notion of groupoid action introduced by Caenepeel and
De Groot in \cite{CG} is  equivalent to this previous one
(see Proposition \ref{prop22}).

Given an action $\beta$ of a finite groupoid $G$ on a $K$-algebra $R$
we can consider the skew groupoid ring $R\star_\beta G$ \cite{BFP}, which
is a $G$-graded algebra or, equivalently, a left $KG^*$-module algebra.
The corresponding weak smash product \cite{CG1} $(R\star_\beta G)\# KG^*$
is a nonunital $K$-algebra (see the subsection 2.5).

Given any unital $K$-algebra $A$ graded by a finite groupoid $G$ or,
equivalently, a left $KG^*$-module algebra,
there exists an action $\beta$ of $G$ on the weak smash product $A\# KG^*$
(see Proposition \ref{prop24})
and we can consider the corresponding skew group ring $(A\# KG^*)\star_\beta G$,
which also is a nonunital $K$-algebra.

\vu

We show in the section 3 (resp., section 4) that $(R\star_\beta G)\# KG^*$
(resp., $(A\# KG^*)\star_\beta G$) contains a unital $K$-subalgebra that is
isomorphic to a finite direct sum of matrix $K$-algebras. In particular,
if $G$ is a group we recover \cite[Theorems 3.2 and 3.5]{CM}.

\vu

In \cite{N} D. Nikshych presented a Blattner-Montgomery-type duality for
weak Hopf algebras, generalizing the well known result for Hopf algebras obtained
in \cite{BM} and \cite{vdB}. There exists a natural relation
between $(R\star_\beta G)\# KG^*$ and the double weak smash product
$(R\bigotimes_{KG_t}KG)\bigotimes_{KG^*_t} KG^*$ as
constructed by Nikshych (in the case that this makes sense), and it will be explicitly
given in the section 5.

\vu

Throughout, by ring (or algebra) we mean an associative, not necessarily commutative and
not necessarily unital ring (or algebra).

\section{Preliminary results}

\subsection{Groupoids}

\vu

The axiomatic version of groupoid that we adopt in this paper was
taken from \cite{Lawson}. A {\it groupoid} is a non-empty set $G$
equipped with a partially defined binary operation, that we will
denote by concatenation, for which the usual axioms of a group hold
whenever they make sense, that is:

\begin{itemize}

\item[(i)] For every $g,h,l\in G$, $g(hl)$ exists if and only if
$(gh)l$ exists and in this case they are equal.

\vu

\item[(ii)] For every $g,h,l\in G$, $g(hl)$ exists if and only if $gh$
and $hl$ exist.

\vu

\item[(iii)] For each $g\in G$ there exist (unique) elements
$d(g),r(g)\in G$ such that $gd(g)$ and $r(g)g$ exist and
$gd(g)=g=r(g)g$.

\vu

\item[(iv)] For each $g\in G$ there exists an element $g^{-1}\in G$
such that $d(g)=g^{-1}g$ and $r(g)=gg^{-1}$.

\end{itemize}

\vu

We will denote by $G^2$ the subset of the pairs $(g,h)\in G\times G$
such that the element $gh$ exists.

\vu

An element $e\in G$ is called an {\it identity} of $G$ if
$e=d(g)=r(g^{-1})$, for some $g\in G$. In this case $e$ is called
the {\it domain identity} of $g$ and the {\it range identity} of
$g^{-1}$. We will denote by $G_0$ the set of all identities of $G$
and we will denote by $G_e$ the set of all $g\in G$ such that
$d(g)=r(g)=e$. Clearly, $G_e$ is a group, called the {\it isotropy
(or principal) group associated to $e$}.

\vu

The assertions listed in the following lemma are straightforward
from the above definition. Such assertions will
be freely used along this paper.

\vu

\begin{lema} \label{lema21}

Let $G$ be a groupoid. Then,

\begin{itemize}

\item[(i)] for every $g\in G$, the element $g^{-1}$ is unique
satisfying $g^{-1}g=d(g)$ and $gg^{-1}=r(g)$,

\vu

\item[(ii)] for every $g\in G$, $d(g^{-1})=r(g)$ and $r(g^{-1})=d(g)$,

\vu

\item[(iii)] for every $g\in G$, $(g^{-1})^{-1}=g$,

\vu

\item[(iv)] for every $g,h\in G$, $(g,h)\in G^2$ if and only if $d(g)=r(h)$,

\vu

\item[(v)] for every $g,h\in G$, $(h^{-1},g^{-1})\in G^2$ if and only if
$(g,h)\in G^2$ and, in this case, $(gh)^{-1}=h^{-1}g^{-1}$,

\vu

\item[(vi)] for every $(g,h)\in G^2$, $d(gh)=d(h)$ and $r(gh)=r(g)$,

\vu

\item[(vii)] for every $e\in G_0$, $d(e)=r(e)=e$ and $e^{-1}=e$,

\vu

\item[(viii)] for every $(g,h)\in G^2$, $gh\in G_0$ if and only if
$g=h^{-1}$,

\vu

\item[(ix)] for every $g,h\in G$, there exists $l\in G$ such that
$g=hl$ if and only if $r(g)=r(h)$,

\vu

\item[(x)] for every $g,h\in G$, there exists $l\in G$ such that
$g=lh$ if and only if $d(g)=d(h)$.
\end{itemize}
\end{lema}

\vu

\subsection{Weak bialgebras: the finite groupoid algebra and its dual}
Hereafter $K$ will denote a unital commutative ring and unadorned
$\otimes$ will mean $\otimes_K$. Following \cite{BNS}, a weak
$K$-bialgebra $H$ is a unital $K$-algebra,
with a $K$-coalgebra structure $(\Delta,\epsilon)$ such that
\begin{itemize}
\item[(i)] $\Delta(xy)=\Delta(x)\Delta(y)$

\vu

\item[(ii)] $\Delta^2(1_H)=(\Delta(1_H)\otimes
1_H)(1_H\otimes\Delta(1_H))=(1_H\otimes\Delta(1_H))(\Delta(1_H)\otimes
1_H)$

\vu

\item[(iii)]
$\varepsilon(xyz)=\sum\varepsilon(xy_1)\varepsilon(y_2z)=\sum\varepsilon(xy_2)\varepsilon(y_1z)$,
\end{itemize}
for all $x,y,z\in H$, where $\Delta^2=(\Delta\otimes
I_H)\circ\Delta=(I_H\otimes\Delta)\circ\Delta$ and $I_H$ denotes the
identity map of $H$. We use the Sweedler-Heyneman notation for the
comultiplication, namely $\Delta(x)=\sum x_1\otimes x_2,$ for all
$x\in H$.

If $H$ is a bialgebra, that is, the maps $\Delta$ and $\varepsilon$
are homomorphisms of algebras, then the above axioms (i)-(iii) are
trivially satisfied.  Here we are concern with the algebra $KG$ of a
groupoid $G$ and its dual in the case that $G$ is finite. Both are
weak bialgebras but not bialgebras.

\vu

Given a groupoid $G$, the groupoid algebra $KG$ is free as a
$K$-module with basis $\{u_g\ |\ g\in G\}$, its multiplication is
given by the rule
\[
u_gu_h=
\begin{cases}
u_{gh} &\text{if}\,\ \ d(g)=r(h)\\
0  &\text{otherwise},
\end{cases}
\]
for all $g,h\in G$, its identity element exists and is $1_{KG}=\sum_{e\in
G_0}u_e$ if and only if $G_0$ is finite \cite{Lun}, and its $K$-coalgebra
structure is given by $$\Delta(u_g)=u_g\otimes u_g\quad\text{and}\quad \varepsilon(u_g)=1_K,$$
for all $g\in G$.

\vu

The dual $KG^\ast$ of $KG$, as a $K$-module, is free with dual basis
$\{v_g\ |\ g\in G\}$, that is, $v_g(u_h)=\delta_{g,h}1_K$, for all
$g,h\in G$. If $G$ is finite, its $K$-algebra structure is given by
$$v_gv_h=\delta_{g,h}v_g\quad\text{and}\quad\sum_{g\in G}v_g=1_{KG^\ast},$$ and its
$K$-coalgebra structure is given by
$$\Delta(v_g)=\sum_{hl=g}v_h\otimes
v_l=\sum_{d(h)=d(g)}v_{gh^{-1}}\otimes
v_h\quad\text{and}\quad\varepsilon(\sum_{g\in G}a_gv_g)=\sum_{e\in
G_0}a_e.$$

\vu

\subsection{Groupoid actions}

Let $G$ be a groupoid and $R$ a not necessarily unital ring. Following \cite{BP}, an action
of $G$ on $R$ is a pair
$$\beta=(\{E_g\}_{g\in G}, \{\beta_g\}_{g\in G}),$$ where for each
$g\in G$, $E_g=E_{r(g)}$ is an ideal of $R$,
$\beta_g:E_{g^{-1}}\to E_g$ is an isomorphism of rings, and the following conditions hold:
\begin{itemize}
\item[(i)] $\beta_e$ is the identity map $I_{E_e}$ of $E_e$, for all $e\in G_0$,

\vu

\item[(ii)] $\beta_g\beta_h(r)=\beta_{gh}(r)$, for all $(g,h)\in G^2$ and
$r\in E_{h^{-1}}=E_{(gh)^{-1}}$.
\end{itemize}

\vu

\nd In particular, $\beta$ induces an action of the group $G_e$ on $E_e$, for every $e\in G_0$.

\vd

In \cite{CG} Caenepeel and De Groot developed a Galois theory for
weak bialgebra actions on algebras. In particular, they
considered the situation where the weak bialgebra is a finite
groupoid algebra and a notion of groupoid action was introduced.
Actually, this later notion and the one above defined, under some additional conditions,
are equivalent, as we will see in the next proposition.

\vu

Following \cite[section 4]{CG}, a $KG$-module algebra is a unital $K$-algebra $R$,
with a left $KG$-module structure given by $\cdot:KG\otimes R\to R$, $u_g\otimes x\mapsto u_g\cdot x$,
such that:
$$u_g\cdot (xy)=(u_g\cdot x)(u_g\cdot y)\quad\text{and}\quad u_g\cdot 1_R=u_{r(g)}\cdot 1_R,$$
for all $x, y\in R$ and $g\in G$.

\vd

\begin{prop}\label{prop22} Let $G$ be a groupoid such that $G_0$ is finite, and $R$ be a unital $K$-algebra. Then
the following statements are equivalent:
\begin{itemize}

\vu

\item[(i)] There exists an action $\beta=(\{E_g\}_{g\in G},
\{\beta_g\}_{g\in G})$ of $G$ on $R$ such that every $E_e$, $e\in
G_0$, is unital and $R=\bigoplus\limits_{e\in G_0}E_e.$

\vu

\item[(ii)] $R$ has an structure of $KG$-module algebra.
\end{itemize}
\end{prop}

\begin{proof}\, (i)$\Rightarrow$(ii)\,  Let $1_g$ denote the
identity element of $E_g$, for every $g\in G$.
It easily follows from the assumptions that each $1_g$ is
a central idempotent of $R$, $1_R=\sum_{e\in G_0} 1_e$ and
$E_e\bigcap E_f=0=E_eE_f$, for all $e\neq f$.

Consider now the action of $KG$ over $R$ given by
$$u_g\cdot r=\beta_g(r1_{g^{-1}}),$$ for every $g\in G$ and $r\in R$.
Such an action induces on $R$ an structure of $KG$-module. Indeed,
$$1_{KG}\cdot r=\sum_{e\in G_0}u_e\cdot r=\sum_{e\in G_0}\beta_e(r1_e)=
\sum_{e\in G_0} r1_e=r\sum_{e\in G_0}1_e=r1_R=r,$$ for all $r\in R$.
Furthermore, it follows from the items (ii),(iv) and (vi) of Lemma \ref{lema21}
that $E_{(gh)^{-1}}=E_{h^{-1}}$,
$E_h=E_{g^{-1}}$, $E_g=E_{gh}$, for all $(g,h)\in G^2$. Hence,
\[
\begin{array}{ccl}
u_g\cdot(u_h\cdot r)&=&\beta_g(\beta_h(r1_{h^{-1}})1_{g^{-1}})=
\beta_g(\beta_h(r1_{h^{-1}}))\\
&=&\beta_{gh}(r1_{h^{-1}})=
\beta_{gh}(r1_{(gh)^{-1}})\\
&=&u_{gh}\cdot r=u_gu_h\cdot r,\
\end{array}
\]
for all $r\in R$ and $(g,h)\in G^2$. For $(g,h)\not\in G^2$ it is
trivial to check that $u_g\cdot (u_h\cdot r)=0=u_gu_h\cdot r$.
So, $R$ is a left $KG$-module.

\vu

Now, since $$u_g\cdot (rs)=\beta_g(rs1_{g^{-1}})=
\beta_g(r1_{g^{-1}})\beta_g(s1_{g^{-1}})=(u_g\cdot r)(u_g\cdot s)$$

\nd and $$u_g\cdot 1_R=\beta_g(1_R1_{g^{-1}})
=1_g=1_{r(g)}=\beta_{r(g)}(1_R1_{r(g)^{-1}})=u_{r(g)}\cdot 1_R,$$
for all $g\in G$ and $r,s\in R$, the required follows.

\vt

(ii)$\Rightarrow$(i)\, Put $1_g=u_g\cdot 1_R$ and $E_g=R1_g$, for
every $g\in G$. So, $1_g=1_{r(g)}$ and by  \cite[Proposition 4.1]{CG}
these elements are central orthogonal idempotents in $R$, and
$R=\bigoplus_{e\in G_0} E_e$. Clearly, each $E_g$, $g\in G$, is an ideal of $R$ and a unital
ring.  Let $\beta_g:E_{g^{-1}}\to E_g$ given by $\beta_g(r)=u_g\cdot r$,
for every $r\in E_{g^{-1}}$ and $g\in G$. It is immediate from the assumptions that $\beta_g$
is a well defined isomorphism of rings with $\beta_g^{-1}=\beta_{g^{-1}}$. Furthermore,
$$\beta_e(r)=\sum_{e'\in G_0}\beta_{e'}(r1_{e'})=
\sum_{e'\in G_0}u_{e'}\cdot r=1_{KG}\cdot r=r,$$ for every $r\in E_{e^{-1}}=E_e$,
and
$$\beta_g(\beta_h(r))=u_g\cdot (u_h\cdot r)=u_{gh}\cdot(r)=\beta_{gh}(r),$$
for every $(g,h)\in G^2$ and $r\in E_{h^{-1}}=E_{(gh)^{-1}}$. The proof is complete.
\end{proof}

\vu

\subsection{The skew groupoid ring}

Let $R$, $G$ and $\beta=(\{E_g\}_{g\in G}, \{\beta_g\}_{g\in G})$ be as in the previous
subsection.
Accordingly \cite[Section 3]{BFP}, the skew groupoid ring
$R\star_\beta G$ corresponding to $\beta$ is defined as the direct
sum
$$R\star_\beta G=\bigoplus_{g\in G}E_g\delta_g$$ in which the
$\delta_g$'s are symbols, with the usual addition, and
multiplication determined by the rule
\[
(x\delta_g)(y\delta_h)=
\begin{cases}
x\beta_g(y)\delta_{gh} &\text{if $(g,h)\in G^2$}\\
0  &\text{otherwise},
\end{cases}
\]
for all $g, h\in G$, $x\in E_g$ and $y\in E_h$.

\vu

This multiplication is well defined. Indeed, if $(g,h)\in G^2$ then $d(g)=r(h)$
(see Lemma \ref{lema21}(iv)). So, $E_{g^{-1}}=E_{r(g^{-1})}=E_{d(g)}=E_{r(h)}=E_h$,
$\beta_g(y)$ makes sense, and $x\beta_g(y)\in E_g=
E_{r(g)}\overset{(\star)}{=}E_{r(gh)}=E_{gh}$, where the equality $(\star)$ is ensured
by Lemma \ref{lema21}(vi).

\vd

By a routine calculation one easily sees that $A=R\star_\beta G$ is associative,
and by \cite[Proposition 3.3]{BFP} it is unital if $G_0$ is finite and $E_e$ is unital for all
$e\in G_0$.
In this case the identity element of $A$ is $1_A=\sum_{e\in G_0}1_e\delta_e$,
where $1_e$ denotes the identity element of $E_e$, for all $e\in G_0$.

\vu

\begin{obs}\label{obs23} Proposition 3.3 in \cite{BFP} asserts that $A=R\star_\beta G$ is unital
if and only if $G_0$ finite. Unfortunately, the existence of the identity element $1_A$ does not
necessarily imply  that $G_0$ is finite, as we will see in the next subsection (note that, in particular,
$A$ is a $G$-graded algebra by construction).
\end{obs}

\vu

\subsection{Groupoid gradings, coations and weak smash products}

Let $G$ be a groupoid and $A$ a not necessarily unital $K$-algebra. We say that $A$
is a  $G$-graded algebra if there exists a family  $\{A_g\}_{g\in G}$ of $K$-submodules
of $A$ such that $$A=\bigoplus_{g\in G}A_g$$ and
\[
A_gA_h
\begin{cases}
\subseteq A_{gh} &\text{if $(g,h)\in G^2$}\\
0  &\text{otherwise},
\end{cases}
\]
for all $g,h\in G$. It easily follows from this definition that each $A_e$, $e\in G_0$, is a
subalgebra of $A$.

\vu

\begin{obs}\label{obs24}
If $A$ is assumed to be unital, then it follows from
\cite[Propositions 2.2 and 2.3, and Remark 2.4]{LL} that:

\begin{itemize}

\item[(i)] the set $J_0=\{e\in G_0\ \mid\ A_e\neq 0\}$ is finite,

\item[(ii)] $A_g=0$ for all $g\in G$ such that either $d(g)$
or $r(g)$ does not belong to $J_0$,

\item[(iii)] every $A_e$, with $e\in J_0$, is unital,

\item[(iv)] $1_A=\sum_{e\in J_0}1_e$, with $1_e$ denoting the identity element
of $A_e$, for all $e\in J_0$.

\item[(v)] for every $e\in J_0$ and $g\in G$ such that $r(g)=e$ (resp., $d(g)=e$),
$1_ea_g=a_g$ (resp., $a_g1_e=a_g$), for all $a_g\in A_g$.
\end{itemize}
\end{obs}

\vu

Recall from \cite{Böhm} and \cite{CG1} that a right $H$-comodule algebra $A$ ($H$ being
a weak bialgebra) is a unital $K$-algebra with a right $H$-comodule structure given by the coaction
$\rho:A\to A\otimes H$  such that $\rho(ab)=\rho(a)\rho(b)$, for all $a,b\in A$, and
$(\rho\otimes I_H)\circ\rho(1_A)=(\rho(1_A)\otimes1_H)(1_A\otimes\Delta(1_H))$

\vd

In all what follows we will assume that $A$ is a unital $K$-algebra and $G$ is finite.

\vu

The existence of a $G$-grading on $A$ is equivalent to
say that $A$ has an structure of a right $KG$-comodule algebra (see \cite[Proposition 3.1]{CG}),
with coation given by $\rho(a)=\sum_{g\in G}a_g\otimes u_g$, for all
$a=\sum_{g\in G}a_g\in A$. This is also equivalent to say that $A$ is a $KG^*$-module algebra,
with the action given by $v_h\cdot a=a_h$, for all $a\in A$ and $h\in G$.

\vu

Hence, we can consider the weak smash product $A\# KG^*$ of $A$ by $KG^*$
(see \cite[section 3]{CG1}), which is equal to $A\otimes KG^*$ as $K$-modules
and the multiplication is given by the following rule:
\[
(a\# v_g)(b\#v_h)=
\begin{cases}
a(v_{gh^{-1}}\cdot b)\#v_h &\text{if $d(h)=d(g)$}\\
0  &\text{otherwise}.
\end{cases}
\]
A routine calculation easily shows that such a multiplication is associative.
Also, $A\# KG^*$ is not unital. To see this it is enough to verify that
any element of the type $x=b\# v_h$, with $b\in A_k$ and $d(k)\neq r(h)$, is a right annhilator
of $A\# KG^*$. Indeed, in this case $gh^{-1}\neq k$ by Lemma \ref{lema21}(x) and therefore
$v_{gh^{-1}}\cdot b=0$ which implies $(a\# v_g)x=0$ for all $a\in A$ and $g\in G$.

\vu

The element $u=1_A\#1_{KG^*}$ is a preunit of $A\# KG^*$, that is, $ux=xu=xu^2$,
for all $x\in A\# KG^*$ (see \cite[section 3]{CG1}).

\vd

In the sequel we will show that there also exists an action $\beta$ of the
groupoid $G$ on $A\# KG^*$, which allows us to construct the skew groupoid
ring $(A\# KG^*)\star_\beta G$.

\vu

Put $B=A\# KG^*$ and, for each $g\in G$, let $$E_g=\bigoplus\limits_{{l,k\in G}\atop
{d(k)=r(g)}}A_l\#v_k$$ and $$\beta_g:E_{g^{-1}}\to E_g\quad\text{given by}\quad
\beta_g(a_l\#v_k)=a_l\#v_{kg^{-1}}.$$

\vd

\begin{prop}\label{prop24} The pair $\beta=(\{E_g\}_{g\in G}, \{\beta_g\}_{g\in G})$ is
an action of $G$ on $B$, and $B=\bigoplus\limits_{e\in G_0}E_e$.
\end{prop}

\begin{proof}\, \, First, it is clear that $E_g=E_{r(g)}$, for all $g\in G$.
Now, taking $x=a_l\# v_k\in E_g$ and $y=b_s\# v_t\in B$, we have by definition
\[
xy=(a_l\# v_k)(b_s\#v_t)=
\begin{cases}
a_l(v_{kt^{-1}}\cdot b_s)\#v_t &\text{if $d(k)=d(t)$}\\
0  &\text{otherwise},
\end{cases}
\]
and $v_{kt^{-1}}\cdot b_s\neq 0$ if and only if $kt^{-1}=s$ or, equivalently, $t=s^{-1}k$.
Thus,
\[
xy=
\begin{cases}
a_lb_s\#v_t &\text{if $t=s^{-1}k$}\\
0  &\text{otherwise},
\end{cases}
\]
Since $d(t)=d(s^{-1}k)=d(k)=r(g)$, it follows that $xy\in E_g$. Similarly, we also have $yx\in E_g$.
Hence, $E_g$ is an ideal of $B$.

\vu

It is immediate to check that $\beta_g:E_{g^{-1}}\to E_g$ is a well defined additive map,
$\beta_{g^{-1}}=\beta_g^{-1}$ for all $g\in G$, and $\beta_e=I_{E_e}$ for all $e\in G_0$.
From the above we also have
\[
\beta_g(xy)=
\begin{cases}
a_lb_s\#v_{tg^{-1}} &\text{if $t=s^{-1}k$}\\
0  &\text{otherwise},
\end{cases}
\]
for all $x=a_l\# v_k$ and $y=b_s\# v_t$ in $E_{g^{-1}}$.
On the other hand, since $d(kg^{-1})=d(g^{-1})=d(tg^{-1})$ (see Lemma \ref{lema21}(vi)), and
$d(k)=r(g^{-1})=d(g)$ we have
\[
\begin{array}{ccl}
\beta_g(x)\beta_g(y)&=&a_l(v_{kg^{-1}(tg^{-1})^{-1}}\cdot b_s)\#v_{tg^{-1}}\\
&=&a_l(v_{kg^{-1}gt^{-1}}\cdot b_s)\#v_{tg^{-1}}\\
&=&a_l(v_{(k(d(g))t^{-1}}\cdot b_s)\#v_{tg^{-1}}\\
&=&a_l(v_{kt^{-1}}\cdot b_s)\#v_{tg^{-1}}\\
&=&a_lb_s\#v_{tg^{-1}}
\end{array}
\]
if $t=s^{-1}k$ and $0$ otherwise. So, $\beta$ is multiplicative.

\vu

Finally, notice that, for all $(g,h)\in G^2$, $\text{dom}(\beta_g\beta_h)=
\beta_{h^{-1}}(E_h\bigcap E_{g^{-1}})= \beta_{h^{-1}}(E_h)=E_{h^{-1}}=
E_{d(h)}=E_{d(gh)}=E_{(gh)^{-1}}=\text{dom}(\beta_{gh})$, and
\[
\begin{array}{ccl}
\beta_g\beta_h(x)&=&\beta_g(\beta_h(a_l\#v_k))=\beta_g(a_l\#v_{kh^{-1}})\\
&=&a_l\#v_{k(gh)^{-1}}=\beta_{gh}(a_l\#v_k)=\beta_{gh}(x),\
\end{array}
\]
for all $x=a_l\#v_k\in E_{h^{-1}}$. Hence, $\beta$ is an action of $G$ on $B$.

\vu

The last assertion is immediate, since $B=\sum_{g\in G}E_{r(g)}$ and
$E_e\bigcap\sum\limits_{{e'\in G_0}\atop{e'\neq e}}E_{e'}=0$.
\end{proof}

\vd

The skew groupoid ring $B\star_\beta G$ is clearly associative (see the
previous subsection). However, it is not unital, because any element of the type
$x=(a_j\# v_k)\delta_s$, with $d(k)=r(s)$ and $r(k)\neq d(j)$, is a left annhilator
of $B\star_\beta G$. Indeed, it enough to verify that $xE_g=0$, for all $g\in G$.
Let $y=a_l\#v_h\in E_g$, that is, $d(h)=r(g)$. Clearly
$xy=0$ if $d(s)\neq r(g)$, and, otherwise, we have
\[
\begin{array}{ccl}
xy&=&(a_j\#v_k)\beta_s(a_l\#v_h)\delta_{sg}\\
&=&(a_j\#v_k)(a_l\#v_{hs^{-1}})\delta_{sg}\\
&=&(a_j(v_{ksh^{-1}}\cdot a_l)\#v_{hs^{-1}})\delta_{sg}\\
\end{array}
\]
But, $v_{ksh^{-1}}\cdot a_l\neq 0$ if and only if $ksh^{-1}=l$, which implies
$r(l)=r(k)$ and $a_j(v_{ksh^{-1}}\cdot a_l)=a_ja_l=0$ because $d(j)\neq r(l)$.

\vd

\section{Duality for groupoid actions}

In this section $R$ will denote a not necessarily unital $K$-algebra, $G$ a
finite groupoid and $\beta=(\{E_g\}_{g\in G}, \{\beta_g\}_{g\in G})$
an action of $G$ on $R$ such that $E_e$ is unital, for all $e\in G_0$.

\vu

We are concerned with a Cohen-Montgomery-type duality theorem \cite{CM} for
groupoid actions. More precisely, we show that the weak smash product
$(R\star_\beta G)\# KG^*$ (which is not unital) contains a unital subalgebra
isomorphic to a finite direct sum  of matrix $K$-algebras. In particular, if
$G$ is a group we recover Theorem 3.2 of \cite{CM}.

\vu

Let $B=(R\star_\beta G)\# KG^*$. The subalgebra
that we are looking for is $$B_0=\bigoplus\limits_{d(g)=r(g)=r(h)}E_g\delta_g\#v_h,$$
and, in order to get our main result, the strategy is to obtain a decomposition
of $B_0$ into a direct sum of suitable unital subalgebras satisfying the
conditions of the following lemma due to D. S. Passman
(see \cite[Lemma 1.6, p. 228]{Pass}).

\vu

\begin{lema}\label{lema31}
Let $S$ be a unital ring, and $1_S=e_1+\cdots +e_n$ be a
decomposition of $1_S$ into a sum of orthogonal idempotents.
Let $U$ be a subgroup of the group of the units of $S$,
and assume that $U$ permutes the set $\{e_1,\ldots, e_n\}$
transitively by conjugation.
Then $S\simeq M_n(e_1Se_1)$.
\end{lema}

The next lemmas are the necessary preparation
to get our purpose. For each $e\in G_0$, put
$S_e=\{g\in G\ | \ \ d(g)=e\}$ and
$T_e=\{g\in G\ |\ r(g)=e\}$.
Notice that $G_e=S_e\bigcap T_e$, for all $e\in G_0$.

\vd

\begin{lema}\label{lema32}
The following statements hold:
\begin{itemize}

\item[(i)] $B_0$ is a unital subalgebra of $B$,
with identity element $w=\sum\limits_{{e\in G_0}
\atop{g\in T_e}}1_e\delta_e\# v_g$.

\vu

\item[(ii)] $W=\{w_e:=1_e\delta_e\#\sum\limits_{g\in T_e}
v_g\ |\ e\in G_0\}$ is a set of central orthogonal idempotents of $B_0$.

\vu

\item[(iii)] $B_0=\bigoplus_{e\in G_0}B_e$,
where $B_e=B_0w_e=\sum\limits_{{g\in G_e}\atop {h\in T_e}}E_g\delta_g\#v_h$
is a unital ideal (so, a subalgebra) of $B_0$ with identity element $w_e$.

\vu

\item[(iv)] For each $e\in G_0$, $W_e=\{w_{e,g}:=1_e\delta_e\#
v_g\ |\ g\in T_e\}$ is a set of noncentral
orthogonal idempotents of $B_e$ whose sum is $w_e$.
\end{itemize}
\end{lema}

\begin{proof}\, \, (i) It is clear that $B_0$ is a $K$-submodule of $B$.
Given $x=a_g\delta_g\# v_h$ and $y=b_l\delta_l\# v_k$ in $B_0$ we have
\[
xy=
\begin{cases}
(a_g\delta_g)(v_{hk^{-1}}\cdot b_l\delta_l)\# v_k &\text{if $d(k)=d(h)$}\\
0  &\text{otherwise},
\end{cases}
\]
and, consequently,
\[
xy=
\begin{cases}
(a_g\delta_g)( b_l\delta_l)\# v_k &\text{if $k=l^{-1}h$}\\
0  &\text{otherwise},
\end{cases}
\]
Since $d(g)=r(h)=d(l^{-1})=r(l)$, it follows that $$(a_g\delta_g)(b_l\delta_l)=
a_g\beta_g(b_l)\delta_{gl}\in E_{gl}\delta_{gl}.$$ So, $xy\in B_0$ because
$r(gl)=r(g)=r(h)=r(l)=r(k)$ and $d(gl)=d(l)=r(k)$.

\vu

Also,
\[
\begin{array}{ccl}
xw&=&\sum\limits_{e\in G_0}(a_g\delta_g\#v_h)(1_e\delta_e\#\sum\limits_{l\in T_e}v_l)\\
&=&\sum\limits_{e\in G_0}\sum\limits_{{l\in T_e}\atop{d(l)=d(h)}}a_g
\delta_g(v_{hl^{-1}}\cdot1_e\delta_e)\# v_l.\\
&=&(a_g\delta_g)(1_{r(h)}\delta_{r(h)})\#v_h\\
&=&(a_g\delta_g)(1_{d(g)}\delta_{d(g)})\#v_h\\
&=&a_g\beta_g(1_{d(g)})\delta_{gd(g)}\#v_h\\
&=&a_g\delta_g\#v_h\\
&=&x,\
\end{array}
\]

\vu

\nd and
\[
\begin{array}{ccl}
wx&=&\sum\limits_{e\in G_0}\sum\limits_{l\in T_e}(1_e\delta_e\#v_l)(a_g\delta_g\#v_h)\\
&=&\sum\limits_{e\in G_0}\sum\limits_{{l\in T_e}\atop{d(l)=d(h)}}
1_e\delta_e(v_{lh^{-1}}\cdot a_g\delta_g)\#v_h\\
&=&(1_{r(g)}\delta_{r(g)})(a_g\delta_g)\#v_h\\
&=&\beta_{r(g)}(a_g)\delta_{r(g)g}\#v_h\\
&=&a_g\delta_g\#v_h\\
&=&x.\
\end{array}
\]

\vt

(ii)\, Let $e,f\in G_0$, $w_e=1_e\delta_e\#\sum\limits_{g\in T_e}v_g$ and
$w_f=1_f\delta_f\#\sum\limits_{h\in T_f}v_h$. Then,
$$w_ew_f=\sum\limits_{g\in
T_e}\sum\limits_{d(k)=d(g)}1_e\delta_e(v_{gk^{-1}}\cdot1_f\delta_f)\#v_k(\sum\limits_{h\in
T_f}v_h),$$ and so
\[
w_ew_f=
\begin{cases}
\sum\limits_{g\in T_e}\sum\limits_{{h\in T_f}\atop {d(h)=d(g)}}
1_e\delta_e(v_{gh^{-1}}\cdot1_f\delta_f)\#v_h &\\
0 &\text{ if}\,\, \, d(h)\neq d(g),\,\, \forall g\in T_e, h\in T_f.
\end{cases}
\]

\vd

\nd Noticing that $v_{gh^{-1}}.1_f\delta_f\neq 0$ if and only if $h=g$ and $e=r(g)=f$,
we have $w_ew_e=\sum\limits_{g\in
T_e}(1_e\delta_e)(1_e\delta_e)\#v_g=1_e\delta_e\#\sum\limits_{g\in T_e}v_g=w_e.$
Therefore,
\[
w_ew_f=
\begin{cases}
w_e &\text{if $e=f$}\\
0 &\text{ otherwise}.
\end{cases}
\]

\vt

It remains to show that $w_e$ is central in $B_0$. Take $x=a_g\delta_g\#v_h\in B_0$. Then,
$$xw_e=(a_g\delta_g\#v_h)(1_e\delta_e\# \sum\limits_{l\in T_e}v_l)=
\sum\limits_{d(t)=d(h)}a_g\delta_g(v_{ht^{-1}}\cdot1_e\delta_e)\#
v_t(\sum\limits_{l\in T_e}v_l).$$ Observe that if $r(t)\neq e$, for all $t\in G$
such that $d(t)=d(h)$, then $v_t(\sum\limits_{l\in T_e}v_l)=0$. Hence,

\[
xw_e=
\begin{cases}
\sum\limits_{{l\in T_e}\atop {d(l)=d(h)}}a_g\delta_g(v_{hl^{-1}}.1_e\delta_e)\# v_l &\\
0 & \mbox{if}\,\, \,d(h)\neq d(l), \,\,\, \forall l\in T_e.
\end{cases}
\]

\vd

\nd Since $v_{hl^{-1}}\cdot1_e\delta_e\neq 0$ if and only if $l=h$, and $e=r(h)=d(g)$, it follows that
$$\sum\limits_{{l\in T_e}\atop {d(l)=d(h)}}a_g\delta_g(v_{hl^{-1}}\cdot1_e\delta_e)\# v_l=
(a_g\delta_g)(1_e\delta_e)\#v_h=a_g\beta_g(1_e)\delta_{ge}\#v_h=a_g\delta_g\#v_h=x.$$  Hence,
\[
xw_e=
\begin{cases}
x &\text{if $h\in T_e$}\\
0 &\text{ otherwise}.
\end{cases}
\]

\vu

\nd By a similar calculation one obtains $w_ex=xw_e$.

\vt

(iii) It easily follows from (ii).

\vd

(iv)\,\, Take $w_{e,g}=1_e\delta_e\#v_g$ and $w_{e,h}=1_e\delta_e\#v_h$ in $W_e$.
Then,
\[
w_{e,g}w_{e,h}=
\begin{cases}
1_e\delta_e(v_{gh^{-1}}\cdot1_e\delta_e)\#v_h
 &\text{if $d(h)=d(g)$}\\
0 &\text{ otherwise}.
\end{cases}
\]
Since $v_{gh^{-1}}\cdot1_e\delta_e\neq 0$ if and only if $h=g$, and $r(g)=e$ by assumption, it follows that
\[
w_{e,g}w_{e,h}=
\begin{cases}
w_{e,g} &\text{if $h=g$}\\
0 &\text{ otherwise}.
\end{cases}
\]

\vu

To see that each $w_{e,l}=1_e\delta_e\#v_l$ is not central in $B_e$,
take $x=a_g\delta_g\#v_h\in B_e$. Then, $r(g)=d(g)=r(h)=e$ and
\[
xw_{e,l}=
(a_g\delta_g\#v_h)(1_e\delta_e\#v_l)=
\begin{cases}
a_g\delta_g(v_{hl^{-1}}\cdot1_e\delta_e)\#v_l &\mbox{ if } d(l)=d(h)\\
0 &\mbox{ otherwise},
\end{cases}
\]
that easily implies
\[
xw_{e,l}=
\begin{cases}
x &\text{if $l=h$}\\
0 &\text{ otherwise}.
\end{cases}
\]

\vu

On the other hand,
\[
w_{e,l}x=(1_e\delta_e\#v_l)(a_g\delta_g\#v_h)=
\begin{cases}
1_e\delta_e(v_{lh^{-1}}\cdot a_g\delta_g)\#v_h &\mbox{ if } d(l)=d(h)\\
0 &\mbox{ otherwise}
\end{cases}
\]
and, consequently,
\[
w_{e,l}x=
\begin{cases}
x &\mbox{ if } l=gh\\
0& \mbox{ otherwise}.
\end{cases}
\]

\vu

The proof is complete.
\end{proof}

\vu

\begin{lema}\label{lema33}
For each $e\in G_0$,
\begin{itemize}

\item[(i)] $U_e=\{u_{g,e}:=1_g\delta_g\#\sum\limits_{h\in T_e}
v_h\ |\ g\in G_e\}$ is a subgroup
of the group of the units of $B_e$,

\vu

\item[(ii)] $U_e$ acts on $W_e$ by conjugation,

\vu

\item[(iii)] $w_{e,h}^{U_e}=\{w_{e,gh}\ |\ g\in G_e\}=
\{w_{e,l}\ |\ l\in S_{d(h)}\}$ is the orbit of
the element $w_{e,h}\in W_e$, under the action of $U_e$ on $W_e$.
\end{itemize}
\end{lema}

\begin{proof}
(i) Let $u_{g,e}=1_g\delta_g\#\sum\limits_{h\in T_e}v_h$
and $u_{l,e}=1_l\delta_l\#\sum\limits_{k\in T_e}v_k$ be elements of $U_e$. Then,


\begin{align*}
u_{g,e}u_{l,e}&=\sum\limits_{h\in
T_e}(1_g\delta_g\#v_h)(1_l\delta_l\#\sum\limits_{k\in T_e}v_k)\\
&=\sum\limits_{h\in T_e}\sum\limits_{{k\in T_e}\atop
{d(k)=d(h)}}1_g\delta_g(v_{hk^{-1}}\cdot1_l\delta_l)\#v_k\\
&=\sum\limits_{h\in
T_e}1_g\delta_g1_l\delta_l\#v_{l^{-1}h}\\
&=1_{gl}\delta_{gl}\#\sum\limits_{h\in T_e}v_{l^{-1}h}
\end{align*}
The last equality follows from the fact that $1_g=1_{r(g)}=1_e=1_{r(l)}=1_l$ and $1_{gl}=1_{r(gl)}=1_{r(g)}=1_g$.
Since $r(l^{-1}h)=r(l^{-1})=d(l)=e$ and $gl\in G_e$, it follows that $u_{g,e}u_{l,e}\in U_e$.

\vu

Finally, taking
$u_{g^{-1},e}=1_{g^{-1}}\delta_{g^{-1}}\#\sum\limits_{l\in T_e}v_l\in U_e$ we have


\begin{align*}
u_{g,e}u_{g^{-1},e}&=(1_g\delta_g\#\sum\limits_{h\in
T_e}v_h)(1_{g^{-1}}\delta_{g^{-1}}\#\sum\limits_{l\in T_e}v_l)\\
&=\sum\limits_{h\in T_e}\sum\limits_{{l\in T_e\atop d(l)=
d(h)}}1_g\delta_g(v_{hl^{-1}}\cdot1_{g^{-1}}\delta_{g^{-1}})\#v_l\\
&=\sum\limits_{h\in T_e}1_g\delta_g1_{g^{-1}}\delta_{g^{-1}}\#v_{h}
\quad (\text{since}\quad 1_{r(h)}=1_e=1_{d(g)}=1_{g^{-1}})\\
&=\sum\limits_{h\in T_e}1_{r(g)}\delta_{r(g)}\#v_{h}\\
&=1_e\delta_e\#\sum\limits_{h\in T_e}v_{h}\\
&=1_{B_e}.
\end{align*}
By a similar calculation one also gets $u_{g^{-1},e}u_{g,e}=1_{B_e}$.

\vt

(ii) Taking $w_{e,l}=1_e\delta_e\#v_l\in W_e$ and $u_{g,e}=1_g\delta_g\#\sum\limits_{h\in T_e}v_h\in U_e$
we have


\[
\begin{array}{ccl}
u_{g,e}w_{e,l}&=&(1_g\delta_g\#\sum\limits_{h\in T_e}v_h)(1_e\delta_e\#v_l)\\
&=&\sum\limits_{h\in
T_e}\sum\limits_{d(t)=d(h)}1_g\delta_g(v_{ht^{-1}}\cdot1_e\delta_e)\#v_tv_l\\
&=& \sum\limits_{{h\in T_e}\atop{d(h)=d(l)}}1_g\delta_g(v_{hl^{-1}}.1_e\delta_e)\#v_l\\
&=&1_g\delta_g1_e\delta_e\#v_l\\
&=&1_g\delta_g\#v_l.\
\end{array}
\]
Hence,
\begin{align*}
u_{g,e}w_{e,l}u_{g^{-1},e}&=(1_g\delta_g\#v_l)(1_{g^{-1}}\delta_{g^{-1}}\#\sum\limits_{k\in
T_e}v_k)\\
&=\sum\limits_{d(t)=d(l)}1_g\delta_g(v_{lt^{-1}}\cdot1_{g^{-1}}\delta_{g^{-1}})\#v_t(\sum\limits_{k\in
T_e}v_k)\\
&=\sum\limits_{{k\in T_e}\atop
{d(k)=d(l)}}1_g\delta_g(v_{lk^{-1}}\cdot1_{g^{-1}}\delta_{g^{-1}})\#v_k\\
&=1_g\delta_g1_{g^{-1}}\delta_{g^{-1}}\#v_{gl}\\
&=1_{r(g)}\delta_{r(g)}\#v_{gl}\\
&=1_e\delta_e\#v_{gl},
\end{align*}
which belongs to $W_e$ because $r(gl)=r(g)=e$.

\vd

(iii) It easily follows from the proof of (ii).
\end{proof}

\vu

It follows from the above that any two elements of $W_e$, say $w_{e,g}$ and $w_{e,h}$,
are in the same orbit by the action of $U_e$ if and only if $d(g)=d(h)$. Hence, the action
of $U_e$ on $W_e$ in general is not transitive. Nevertheless, $W_e$ contains subsets
$\Omega_{e,h_1},\ldots, \Omega_{e,h_{n_e}}$ and $U_e$ contains subgroups $U_{e,h_1},\ldots, U_{e,h_{n_e}}$
such that each $U_{e,h_i}$ acts transitively on $\Omega_{e,h_i}$ by conjugation, as we shall see
in the sequel.

\vu

\begin{lema}\label{lema34} For each $e\in G_0$, let $w_{e,h_1}^{U_e},\ldots
,w_{e,h_{n_e}}^{U_e}$ be the distinct orbits of $U_e$ in $W_e$, and
$\omega_{e,h_i}$ denote the sum of all elements of the orbit
$w_{e,h_i}^{U_e}$, for each $1\leq i\leq n_e$. Then,

\begin{itemize}
\item[(i)] $d(h_i)\neq d(h_j)$, for all $i\neq j$,

\vu

\item[(ii)] $\omega_{e,h_i}=1_e\delta_e\#\sum\limits_{l\in T_e\cap
S_{d(h_i)}}v_l\,\,$ and $\,\, w_e=\sum\limits_{1\leq i\leq
n_e}\omega_{e,h_i}$,

\vu

\item[(iii)] the elements $\omega_{e,h_i}$ are central orthogonal idempotents of
$B_e$,

\vu

\item[(iv)] for each $1\leq i\leq n_e$,
$B_{e,h_i}:=B_e\omega_{e,h_i}=\bigoplus\limits_{{g\in G_e}\atop
{l\in T_e\bigcap S_{d(h_i)}}}E_g\delta_g\#v_l$ is a unital ideal (hence, a subalgebra) of $B_e$
with identity element $\omega_{e,h_i}$,

\vu

\item[(v)] $B_e=\bigoplus\limits_{1\leq i\leq n_e}B_{e,h_i}$.
\end{itemize}
\end{lema}

\begin{proof}

(i) It follows from Lemma \ref{lema33}(iii).

\vd

(ii) Immediate.

\vd

(iii) Notice that each $\omega_{e,h_i}$ is a sum of orthogonal
idempotents of $W_e$, which are all orthogonal. So, it is immediate to check that
all the $\omega_{e,h_i}$, $1\leq i\leq n_e$, are orthogonal.

\vu

It remains to show that such idempotents are central in $B_e$.

\vu

Take $x=a_l\delta_l\#v_k\in B_e$. Then, $l\in G_e$, $k\in T_e$ and
\begin{align*}
\omega_{e,h_i}x&=(1_e\delta_e\#\sum\limits_{t\in T_e\bigcap
S_{d(h_i)}}v_t)(a_l\delta_l\#v_k)\\
&=\sum\limits_{t\in T_e\bigcap
S_{d(h_i)}}(1_e\delta_e\#v_t)(a_l\delta_l\#v_k)\\
&=\sum\limits_{t\in T_e\bigcap S_{d(h_i)}}\sum\limits_{d(s)=d(t)}
1_e\delta_e(v_{ts^{-1}}\cdot a_l\delta_l)\#v_sv_k\\
&=\begin{cases}
\sum\limits_{t\in T_e\bigcap
S_{d(h_i)}}1_e\delta_e(v_{tk^{-1}}\cdot a_l\delta_l)\#v_k &\text{if $d(k)=d(h_i)$}\\
0 &\text{otherwise}
\end{cases}\\
&=\begin{cases}
a_l\delta_l\#v_k &\text{if $d(k)=d(h_i)$}\\
0 &\text{otherwise},
\end{cases}
\end{align*}
since $v_{tk^{-1}}\cdot a_l\delta_l\neq 0$ if and only if $t=lk$ and, in this case,
$d(t)=d(lk)=d(k)$ and $r(t)=r(lk)=r(l)=e$.

\vu

By a similar calculation one obtains $x\omega_{e,h_i}=\omega_{e,h_i}x$.

\vd

(iv) It follows from (iii).

\vd

(v) It follows from (ii) and (iii).
\end{proof}

\begin{lema}\label{lema35}
For each $e\in G_0$ and $1\leq i\leq n_e$,
\begin{itemize}
\item[(i)]
$\Omega_{e,h_i}=\{\omega_{e,h_i,l}:=1_e\delta_e\#v_l\,\, | \,\,
l\in T_e\cap S_{d(h_i)}\}$ is a set of noncentral orthogonal idempotents of
$B_{e,h_i}$ whose sum is $1_{B_{e,h_i}}=\omega_{e,h_i}$,

\vu

\item[(ii)] the set
$U_{e,h_i}=\{u_{g,e,h_i}:=1_g\delta_g\#\sum\limits_{l\in T_e\bigcap
S_{d(h_i)}}v_l\,\, | \,\, g\in G_e\}$ is a subgroup of $U_e$,

\vu

\item[(iii)] $U_{e,h_i}$ acts transitively on $\Omega_{e,h_i}$ by conjugation.
\end{itemize}
\end{lema}

\begin{proof}
(i) It follows from Lemma \ref{lema32}(iv) and Lemma \ref{lema34}(ii)(iv).

\vd

(ii) Take $u_{l,e,h_i}=1_l\delta_l\#\sum\limits_{k\in T_e\bigcap S_{d(h_i)}}v_k$
and  $u_{t,e,h_i}=1_t\delta_t\#\sum\limits_{s\in T_e\bigcap S_{d(h_i)}}v_s$ in $U_{e,h_i}$. Then,
\begin{align*}
u_{l,e,h_i}u_{t,e,h_i}&=(1_l\delta_l\#\sum\limits_{k\in T_e\bigcap
S_{d(h_i)}}v_k)(1_t\delta_t\#\sum\limits_{s\in T_e\bigcap
S_{d(h_i)}}v_s)\\
&=\sum\limits_{k,s\in T_e\bigcap
S_{d(h_i)}}(1_l\delta_l\#v_k)(1_t\delta_t\#v_s)\\
&=\sum\limits_{k,s\in T_e\bigcap
S_{d(h_i)}}(1_l\delta_l)(v_{ks^{-1}}\cdot1_t\delta_t)\#v_s\\
&=\sum\limits_{s\in T_e\bigcap S_{d(h_i)}}(1_l\delta_l)(1_t\delta_t)\#v_s\\
&=1_{lt}\delta_{lt}\#\sum\limits_{s\in T_e\bigcap
S_{d(h_i)}}v_s,
\end{align*}
because $l\in G_e$, $t\in T_e$ and $1_{l}=1_{r(l)}=1_{r(lt)}=1_{lt}$, which implies
$(1_l\delta_l)(1_t\delta_t)=\beta_l(1_t)\delta_{lt}=\beta_l(1_{r(t)})\delta_{lt}=
\beta_l(1_{d(l)})\delta_{lt}=1_l\delta_{lt}=1_{lt}\delta_{lt}$. And such a product
belongs to $U_{e,h_i}$ because $r(lt)=r(l)=e=d(t)=d(lt)$.

\vu

Finally, it is straightforward to check that $u_{l^{-1},e,h_i}=1_{l^{-1}}\delta_{l^{-1}}\#\sum\limits_{t\in
T_e\bigcap S_{d(h_i)}}v_t$ is the inverse of  $u_{l,e,h_i}$, and clearly it also belongs to $U_{e,h_i}$.

\vt

(iii) Take $u_{g,e,h_i}=1_g\delta_g\#\sum\limits_{l\in T_e\cap
S(d(h_i))}v_l\in U_{e,h_i}$ and
$\omega_{e,h_i,k}=1_e\delta_e\#v_k\in W_{e,h_i}$. By a calculation identical to that done in the proof
of the item (ii) of Lemma \ref{lema33} one gets
$$u_{g,e,h_i}\omega_{e,h_i,k}u_{g^{-1},e,h_i}=1_e\delta_e\#v_{gk},$$ and this action is clearly transitive.
\end{proof}

\vu

\begin{prop}\label{prop36}
$B_{e,h_i}\simeq M_{n_{e,h_i}}(E_e)$ as unital $K$-algebras, for all $e\in G_0$ and $1\leq i\leq n_e$.
\end{prop}

\begin{proof} It follows from Lemmas \ref{lema31} and \ref{lema35} that $B_{e,h_i}\simeq M_{n_{e,h_i}}(C_{e,h_i})$,
where $$C_{e,h_i}=(\omega_{e,h_i,l_1})B_{e,h_i}(\omega_{e,h_i,l_1}),$$ with $l_1\in T_e\bigcap S_{d(h_i)}$. It remains
to prove that $C_{e,h_i}$ is isomorphic to $E_e$. Recalling from Lemma \ref{lema34} that any element of $B_{e,h_i}$ is
of the form $x=\sum\limits_{{g\in G_e}\atop {k\in T_e\cap S_{d(h_i)}}}a_g\delta_g\#v_k$, we have
\begin{align*}
x\omega_{e,h_i,l_1}&=\sum\limits_{{g\in G_e}\atop {k\in T_e\cap
S_{d(h_i)}}}a_g\delta_g(v_{kl_1^{-1}}\cdot1_e\delta_e)\#v_{l_1}\\
&=\sum\limits_{g\in G_e}(a_g\delta_g)(1_e\delta_e)\#v_{l_1}\\
&=\sum\limits_{g\in G_e}a_g\delta_g\#v_{l_1}.
\end{align*}
and
\begin{align*}
\omega_{e,h_i,l_1}x\omega_{e,h_i,l_1}&=(1_e\delta_e\#v_{l_1})(\sum\limits_{g\in
G_e}a_g\delta_g\#v_{l_1})\\
&=\sum\limits_{g\in
G_e}(1_e\delta_e)(v_{l_1l_1^{-1}}\cdot 1_g\delta_g)\#v_{l_1}\\
&=\sum\limits_{g\in G_e}(1_e\delta_e)(v_{e}.a_g\delta_g)\#v_{l_1}\\
&=(1_e\delta_e)(a_e\delta_e)\#v_{l_1}\\
&=a_e\delta_e\#v_{l_1}.
\end{align*}
Hence, $C_{e,h_i}=E_e\delta_e\#v_{l_1}=E_e(1_A\#v_{l_1})$, which is naturally isomorphic to
$E_e$ as $K$-algebras via the map $a_e\mapsto a_e\delta_e\#v_{l_1}=a_e(1_A\#v_{l_1})$.
Observe that this map is surjective by definition and it is straightforward  to check
that it is a homomorphism of $K$-algebras. Its injectivity follows from the freeness
of $1_A\# v_{l_1}$ over $A=R\star_\beta G$.
\end{proof}

\vu

\begin{teo}\label{teo37}

$$B_0\simeq\bigoplus\limits_{e\in
G_0}(\bigoplus\limits_{i=1}^{n_e}M_{n_{e,h_i}}(E_e))$$ as unital $K$-algebras.
\end{teo}

\begin{proof}
It follows from Lemmas \ref{lema32}(iii) and \ref{lema34}(v), and Proposition \ref{prop36}.
\end{proof}

\section{Duality for groupoid coactions}

In all this section $G$ is a finite groupoid and $A$ is a unital $G$-graded $K$-algebra.
Recall from the subsection 2.5 that $A$ is a left $KG^*$-module algebra  via the action
$v_k\cdot\sum_{g\in G}a_g=a_k$, for all $k\in G$, and there exists an action $\beta$ of $G$ on
the corresponding weak smash product $B=A\# KG^*$. Let $C=B\star_\beta G$ be the corresponding
skew groupoid ring. Like in the section 3, also here we obtain a Cohen-Montgomery-type
duality theorem, that is, we show that the $K$-algebra $C$ contains a unital
subalgebra isomorphic to a finite direct sum of matrix $K$-algebras. In particular, if $G$
is a group we recover \cite[Theorem 3.5]{CM}.

\vu

The steps to get our purpose are similar to those in the previous section.
Recall from Proposition \ref{prop24} that $$C=\bigoplus\limits_{g\in
G}{(\bigoplus\limits_{d(k)=r(g)}A_l\#v_k)}\delta_g.$$
Let $$C_0=\bigoplus\limits_{e\in G_0}(\bigoplus\limits_{g\in
G_e}(\bigoplus\limits_{{r(l)=d(l)=r(k)}\atop
{d(k)=e}}A_l\#v_k)\delta_g).$$

\begin{lema}\label{lema41}
The following statements hold:
\begin{itemize}
\item[(i)] $C_0$ is a unital $K$-algebra with identity element
$w=\sum\limits_{e\in G_0}(\sum\limits_{f\in
G_0}1_f\#\sum\limits_{h\in T_f\bigcap S_e}v_h)\delta_e$,

\vu

\item[(ii)] $W=\{w_e:=(\sum\limits_{f\in G_0}1_f\#\sum\limits_{h\in
T_f\bigcap S_e}v_h)\delta_e\,\, | \,\, e\in G_0\}$ is a set of central orthogonal
idempotents of $C_0$, whose sum is $w$,

\vu

\item[(iii)] $C_0=\bigoplus\limits_{e\in G_0}C_e$, where
$C_e=C_0w_e=\bigoplus\limits_{g\in
G_e}(\bigoplus\limits_{{r(l)=d(l)=r(k)}\atop
{d(k)=e}}A_l\#v_k)\delta_g$ is a unital ideal (so, subalgebra) of
$C_0$ with identity element $w_e$.
\end{itemize}
\end{lema}

\begin{proof}
(i) Clearly, $C_0$ is a $K$-submodule of $C$. Now, take $x=(a_l\#v_k)\delta_g$ and
$y=(b_s\#v_t)\delta_h$ in $C_0$. Then, $r(l)=d(l)=r(k)$, $d(k)=e$, $r(s)=d(s)=r(t)$,
$d(t)=f$, $g\in G_e$ and $h\in G_f$, with $e,f\in G_0$. If $e\neq f$ it is immediate that
$xy=0$. Otherwise, we have
\[
\begin{array}{ccl}
xy&=& (a_l\#v_k)\beta_g(b_s\#v_t)\delta_{gh}\\
&=& (a_l\#v_k)(b_s\#v_{gt^{-1}})\delta_{gh}\\
&=&(a_l(v_{k(tg^{-1})}.b_s)\#v_{tg^{-1}})\delta_{gh}\\
&=&(a_lb_s\#v_{tg^{-1}})\delta_{gh},\\
&=&(a_lb_s\#v_{s^{-1}k})\delta_{gh}\
\end{array}
\]
if $s=kgt^{-1}$ or, equivalently, $tg^{-1}=s^{-1}k$.  Observing that
\begin{itemize}
\item[---] $a_lb_s\in A_{ls}$, because $d(l)=r(k)=r(s)$,

\vu

\item[---] $d(ls)=d(s)$, $r(ls)=r(l)=r(k)=r(s)=d(s)$
and $r(s^{-1}k)=r(s^{-1})=d(s)$,

\vu

\item[---] $d(s^{-1}k)=d(k)=e=r(g)=r(gh)$, and

\vu

\item[---] $gh\in G_e$
\end{itemize}
we conclude that $xy\in C_0$. It remains to prove that
$w=\sum\limits_{e\in G_0}(\sum\limits_{f\in
G_0}1_f\#\sum\limits_{h\in T_f\bigcap S_e}v_h)\delta_e$ is the
identity element of $C_0$. Indeed, taking $x=(a_l\#v_k)\delta_g$,
with $g\in G_{e'}$, for some $e'\in G_0$, $r(l)=d(l)=r(k)$
and  $d(k)=e'$, we have
\begin{align*}
wx&=\sum\limits_{e\in G_0}(\sum\limits_{f\in
G_0}1_f\#\sum\limits_{h\in T_f\bigcap
S_e}v_h)\delta_e(a_l\#v_k)\delta_g\\
&=(\sum\limits_{f\in G_0}1_f\#\sum\limits_{h\in T_f\bigcap
S_{e'}}v_h)\delta_{e'}(a_l\#v_k)\delta_g\\
&=(\sum\limits_{f\in G_0}1_f\#\sum\limits_{h\in T_f\bigcap
S_{e'}}v_h)\beta_{e'}(a_l\#v_k)\delta_g\\
&=(\sum\limits_{f\in G_0}\sum\limits_{h\in T_f\bigcap S_{e'}}
(1_f\#v_h)(a_l\#v_k))\delta_g\\
&=(\sum\limits_{f\in G_0}\sum\limits_{h\in T_f\bigcap S_{e'}}
1_f(v_{hk^{-1}}\cdot a_l)\#v_k)\delta_g\\
&=(1_{r(l)}a_l\#v_k)\delta_g.\\
&=(a_l\#v_k)\delta_g\quad(\text{by Remark \ref{obs24}(v)})\\
&=x
\end{align*}
One easily verifies that $xw=x$ by the same way.

\vd

(ii) Let $e,e'\in G_0$. It is clear that $w_ew_{e'}=0$ if $e\neq e'$.
Otherwise we have
\[
\begin{array}{ccl}
w_ew_{e'}&=&[(\sum\limits_{f\in G_0}1_f\#\sum\limits_{h\in
T_f\bigcap S_e}v_h)\delta_e][(\sum\limits_{f'\in G_0}1_{f'}\#
\sum\limits_{l\in T_{f'}\bigcap S_{e'}}v_l)\delta_{e}]\\
&=& (\sum\limits_{f\in G_0}1_f\#\sum\limits_{h\in T_f\bigcap S_e}v_h)
(\sum\limits_{f'\in G_0}1_{f'}\#
\sum\limits_{l\in T_{f'}\bigcap S_e}v_l)\delta_e\\
&=&\sum\limits_{f,f'\in G_0}\sum\limits_{{h\in T_f\bigcap S_e}\atop
{l\in T_{f'}\bigcap S_e}}(1_f\#v_h)(1_{f'}\#v_l)\delta_e\\
&=&\sum\limits_{f,f^{'}\in G_0}\sum\limits_{{h\in T_f\bigcap S_e}\atop
{l\in T_{f'}\bigcap S_e}}1_f(v_{hl^{-1}}\cdot 1_{f'})\#v_l\delta_e\\
&=&\sum\limits_{f\in G_0}1_f\#\sum\limits_{h\in T_f\bigcap S_e}v_h\delta_e\\
&=& w_e\
\end{array}
\]
because $v_{hl^{-1}}\cdot1_{f^{'}}\neq 0$ if and only if $h=l$,
and, in this case, $f=f'$.

\vu

It remains to prove that $w_e$ is central in $C_0$, for all $e\in G_0$.
Let $x=(a_l\#v_k)\delta_g\in C_0$, that is, $g\in G_{e'}$, for some
$e'\in G_0$, $r(l)=d(l)=r(k)$ and $d(k)=e'$. Again, it is clear that
$xw_e=0$ if $e\neq e'$. Otherwise,
\[
\begin{array}{ccl}
xw_e&=&\sum\limits_{f\in G_0}\sum\limits_{h\in T_f\bigcap S_e}(a_l\#v_k)
\beta_g(1_f\#v_h)\delta_{ge}\\
&=&\sum\limits_{f\in G_0}\sum\limits_{h\in T_{f}\bigcap S_e}(a_l\#v_k)
(1_f\#v_{hg^{-1}})\delta_{ge}\\
&=&\sum\limits_{f\in G_0}\sum\limits_{h\in T_f\bigcap S_e}
(a_l(v_{kgh^{-1}}\cdot 1_f)\#v_{hg^{-1}})\delta_{g}\\
&=&(a_l1_{r(k)}\#v_k)\delta_{g}\\
&=&x\
\end{array}
\]
by Remark \ref{obs24}(v). Similarly, one also gets $w_ex=x$.

\vd

(iii) It is immediate from the above.
\end{proof}

\vd

\begin{lema}\label{lema42}
The following statements hold:
\begin{itemize}
\item[(i)] For any $f\in G_0$,
$W_f=\{w_{e,f}:=(1_f\#\sum\limits_{h\in T_f\bigcap
S_e}v_h)\delta_e\,\, | \,\, e\in G_0\}$ is a set of central orthogonal
idempotents of $C_e$, whose sum is $w_e$,

\vu

\item[(ii)] For any  $e\in G_0$, $C_e=\bigoplus\limits_{f\in
G_0}C_{e,f}$ where $C_{e,f}=C_ew_{e,f}=\bigoplus\limits_{g\in
G_e}(\bigoplus\limits_{{l\in G_f}\atop {k\in T_f\bigcap
S_e}}A_l\#v_k)\delta_g$ is a unital ideal (so, subalgebra)
of $C_e$, with identity element $w_{e,f}$.
\end{itemize}
\end{lema}

\begin{proof}
(i) \, For any $e,f,f'\in G_0$ we have
\begin{align*}
w_{e,f}w_{e,f'}&=(1_f\#\sum\limits_{h\in T_f\bigcap
S_e}v_h)\delta_e(1_{f'}\#\sum\limits_{l\in T_{f'}\cap
S_e}v_l)\delta_e\\
&=(1_f\#\sum\limits_{h\in T_f\bigcap
S_e}v_h)\beta_e(1_{f'}\#\sum\limits_{l\in T_{f'}\bigcap
S_e}v_l)\delta_e\\
&=\sum\limits_{h\in T_f\bigcap S_e}\sum\limits_{l\in T_{f'}\bigcap
S_e}(1_f\#v_h)(1_{f'}\#v_{le})\delta_e\\
&=\sum\limits_{h\in T_f\bigcap S_e}\sum\limits_{l\in T_{f'}\bigcap
S_e}(1_f\#v_h)(1_{f'}\#v_{l})\delta_e\\
&=\sum\limits_{h\in T_f\bigcap S_e}\sum\limits_{l\in T_{f'}\bigcap
S_e}(1_f(v_{hl^{-1}}\cdot 1_{f'})\#v_l)\delta_e\\
&=
\begin{cases}
(1_{f}\#\sum\limits_{h\in T_{f}\bigcap S_e}v_h)\delta_e=w_{e,f}  &\mbox{ if } f'=f\\
0 &\mbox{ otherwise}.
\end{cases}
\end{align*}

To show that each $w_{e,f}$ is central in $C_e$ take $x=(a_l\#v_k)\delta_g\in C_e$.
So, $g\in G_e$, $r(l)=d(l)=r(k)$, $d(k)=e$,  and
\begin{align*}
xw_{e,f}&=(a_l\#v_k)\beta_g(1_f\#\sum\limits_{h\in T_f\bigcap
S_e}v_h)\delta_g\\
&=(a_l\#v_k)(1_f\#\sum\limits_{h\in T_f\bigcap
S_e}v_{hg^{-1}})\delta_g\\
&=\sum\limits_{h\in T_f\bigcap
S_e}(a_l\#v_k)(1_f\#v_{hg^{-1}})\delta_g\\
&=\sum\limits_{h\in T_f\bigcap
S_e}(a_l(v_{kgh^{-1}}\cdot 1_f)\#v_{hg^{-1}})\delta_g\\
&=
\begin{cases}
(a_l\#v_k)\delta_g=x &\mbox{ if }f=r(k)\\
0 &\mbox{otherwise},
\end{cases}
\end{align*}
since $v_{kgh^{-1}}.1_f\neq 0$ if and only if $h=kg$ if and only if
$f=r(h)=r(k)$ (Lemma \ref{lema21}(ix)).
One also gets $w_{e,f}x=xw_{e,f}$ in a similar way.

\vd

(ii) It is clear from the above.
\end{proof}

\begin{lema}\label{lema43}
For each  $e,f\in G_0$,

\begin{itemize}
\item[(i)] $W_{e,f}=\{w_{e,f,h}:=(1_f\#v_h)\delta_e\,\, | \,\,
h\in T_f\bigcap S_e\}$ is a set of noncentral orthogonal idempotents, with
sum $w_{e,f}=1_{C_{e,f}}$,

\vu

\item[(ii)] $U_{e,f}=\{u_{e,f,g}:=(1_f\#\sum\limits_{l\in T_f\bigcap
S_e}v_l)\delta_g\,\, | \,\, g\in G_e\}$ is a subgroup of the group
of the units of $C_{e,f}$,

\vu

\item[(iii)] $U_{e,f}$ acts transitively on $W_{e,f}$ by conjugation.
\end{itemize}
\end{lema}

\begin{proof}
(i) Let $e,f\in G_0$  and $h,l\in T_f\bigcap S_e$. Then,
\begin{align*}
w_{e,f,h}w_{e,f,l}&=(1_f\#v_h)\delta_e(1_f\#v_l)\delta_e\\
&=(1_f\#v_h)\beta_e(1_f\#v_l)\delta_e\\\
&=(1_f\#v_h)(1_f\#v_l)\delta_e\\
&=(1_f(v_{hl^{-1}}\cdot 1_f)\#v_l)\delta_e\\
&=
\begin{cases}
(1_f\#v_h)\delta_e=w_{e,f,h} &\mbox{ if }h=l\\
0 &\mbox{otherwise}.
\end{cases}
\end{align*}

Clearly, $\sum\limits_{h\in T_e\bigcap S_e}w_{e,f,h}=w_{e,f}=1_{C_{e,f}}$.
Also, $w_{e,f,h}$ is not central in $C_{e,f}$. Indeed, let
$x=(a_l\#v_k)\delta_g\in C_{e,f}$. So, $g\in G_e$, $l\in G_f$,
$k\in T_f\cap S_e$, and
\begin{align*}
w_{e,f,h}x&=(1_f\#v_h)\delta_e(a_l\#v_k)\delta_g\\
&=(1_f\#v_h)\beta_e(a_l\#v_k)\delta_{eg}\\
&=(1_f\#v_h)(a_l\#v_k)\delta_g\\
&=(1_f(v_{hk^{-1}}\cdot a_l)\#v_k)\delta_g\\
&=
\begin{cases}
x &\mbox{ if }h=lk\\
0 &\mbox{otherwise}.
\end{cases}
\end{align*}

On the other hand,
\begin{align*}
xw_{e,f,h}&=(a_l\#v_k)\delta_g(1_f\#v_h)\delta_e\\
&=(a_l\#v_k)\beta_g(1_f\#v_h)\delta_{ge}\\
&=(a_l\#v_k)(1_f\#v_{hg^{-1}})\delta_g\\
&=(a_l(v_{kgh^{-1}}\cdot1_f)\#v_{hg^{-1}})\delta_g\\
&=
\begin{cases}
x & \mbox{ if }h=kg\\
0 & \mbox{ otherwise},
\end{cases}
\end{align*}
since $kgg^{-1}k^{-1}=kr(g)k^{-1}=kek^{-1}=kk^{-1}=r(k)=f$ and
$a_l1_f=a_l$ by Remark \ref{obs24}(v).

\vd

(ii) Let $e,f\in G_0$  and $g,h\in G_e$. Then,
\begin{align*}
u_{e,f,g}u_{e,f,h}&=(1_f\#\sum\limits_{l\in T_f\bigcap
S_e}v_l)\delta_g(1_f\#\sum\limits_{k\in T_f\bigcap S_e}v_k)\delta_h\\
&=(1_f\#\sum\limits_{l\in T_f\bigcap
S_e}v_l)\beta_g(1_f\#\sum\limits_{k\in T_f\bigcap
S_e}v_k)\delta_{gh}\\
&=(1_f\#\sum\limits_{l\in T_f\bigcap S_e}v_l)(1_f\#\sum\limits_{k\in
T_f\bigcap S_e}v_{kg^{-1}})\delta_{gh}\\
&=(1_f\#\sum\limits_{l\in T_f\bigcap
S_e}v_l)(1_f\#\sum\limits_{t\in T_f\bigcap S_e}v_t)\delta_{gh}\quad
(\text{setting}\,\,\, t=kg^{-1})\\
&=\sum\limits_{l,t\in T_f\bigcap
S_e}(1_f(v_{lt^{-1}}\cdot 1_f)\#v_t)\delta_{gh}\\
&=\sum\limits_{l\in T_f\bigcap S_e}(1_f\#v_l)\delta_{gh},
\end{align*}
and this last term belongs to $U_{e,f}$ because $r(gh)=r(g)=e=d(h)=d(gh)$
(Lemma \ref{lema21}(vi)).

\vu

\nd The equality $u_{e,f,g^{-1}}=u^{-1}_{e,f,g}$ is straightforward.

\vd

(iii) Let $e,f\in G_0$, $g\in G_e$ and $h\in T_f\bigcap S_e$. Then,
\begin{align*}
u_{e,f,g}w_{e,f,h}u_{e,f,g^{-1}}&=(1_f\#\sum\limits_{l\in T_f\bigcap
S_e}v_l)\delta_g(1_f\#v_h)\delta_e(1_f\#\sum\limits_{k\in T_f\bigcap
S_e}v_k)\delta_{g^{-1}}\\
&=(1_f\#\sum\limits_{l\in T_f\bigcap
S_e}v_l)\beta_g(1_f\#v_h)\delta_g(1_f\#\sum\limits_{k\in T_f\bigcap
S_e}v_k)\delta_{g^{-1}}\\
&=\sum\limits_{l\in T_f\bigcap
S_e}(1_f\#v_l)(1_f\#v_{hg^{-1}})\delta_g(1_f\#\sum\limits_{k\in
T_f\bigcap S_e}v_k)\delta_{g^{-1}}\\
&=\sum\limits_{l\in T_f\bigcap
S_e}(1_f(v_{l(hg^{-1})^{-1}}\cdot 1_f)\#v_{hg^{-1}})\delta_g(1_f\#\sum\limits_{k\in
T_f\bigcap S_e}v_k)\delta_{g^{-1}}\\
&=(1_f\#v_{hg^{-1}})\delta_g(1_f\#\sum\limits_{k\in T_f\bigcap
S_e}v_k)\delta_{g^{-1}}\\
&=(1_f\#v_{hg^{-1}})\beta_g(1_f\#\sum\limits_{k\in T_f\bigcap
S_e}v_k)\delta_{gg^{-1}}\\
&=(1_f\#v_{hg^{-1}})(1_f\#\sum\limits_{k\in T_f\bigcap
S_e}v_{kg^{-1}})\delta_e\\
&=\sum\limits_{k\in T_f\bigcap
S_e}(1_f(v_{hg^{-1}(kg^{-1})^{-1}}\cdot 1_f)\#v_{kg^{-1}})\delta_e\\
&=\sum\limits_{k\in T_f\bigcap
S_e}(1_f(v_{hk^{-1}}\cdot1_f)\#v_{kg^{-1}})\delta_e\\
&=(1_f\#v_{hg^{-1}})\delta_e,
\end{align*}
and this last term belongs to $W_{e,f}$ since $r(gh^{-1})=r(h)=f$ and
$d(gh^{-1})=r(g)=e$. One easily sees from the above that this action
of $U_{e,f}$ on $W_{e,f}$ is transitive.
\end{proof}

\vd

\begin{prop}\label{prop44}
Let $e,f\in G_0$ and $h\in T_f\bigcap S_e$. Then,

\[
C_{e,f}\simeq M_{n_{e,f}}(\bigoplus\limits_{g\in G_e}A_g),
\]
as unital $K$-algebras, where $n_{e,f}$ denotes the
cardinality of the orbit of $w_{e,f,h}$.
\end{prop}

\begin{proof}
It follows from Lemmas \ref{lema43} and \ref{lema31} that
$$C_{e,f}\simeq M_{n_{e,f}}(S_{e,f}),$$ where $S_{e,f}=
w_{e,f,h}C_{e,f}w_{e,f,h}$, for some $w_{e,f,h}\in W_{e,f}$.

\vd

Now, note that $S_{e,f}=\bigoplus\limits_{g\in
G_e}(A_{hgh^{-1}}\#v_{hg^{-1}})\delta_g$. Indeed, for any
 $x=(a_l\#v_k)\delta_g\in C_{e,f}$ we have
\begin{align*}
xw_{e,f,h}&=(a_l\#v_k)\delta_g(1_f\#v_h)\delta_e\\
&=(a_l\#v_k)\beta_g(1_f\#v_h)\delta_g\\
&=(a_l\#v_k)(1_f\#v_{hg^{-1}})\delta_g\\
&=(a_l(v_{kgh^{-1}}\cdot 1_f)\#v_{hg^{-1}})\delta_g\\
&=
\begin{cases}
(a_l\#v_{hg^{-1}})\delta_g &\mbox{ if }kg=h\\
0 &\mbox{ otherwise},
\end{cases}
\end{align*}
since $l\in G_f$ and $a_l1_f=a_l$ by Remark \ref{obs24}(v).
Thus, if $kg=h$
\begin{align*}
w_{e,f,h}xw_{e,f,h}&=(1_f\#v_h)\delta_e(a_l\#v_{hg^{-1}})\delta_g\\
&=(1_f\#v_h)\beta_e(a_l\#v_{hg^{-1}})\delta_g\\
&=(1_f\#v_h)(a_l\#v_{hg^{-1}})\delta_g\\
&=(1_f(v_{hgh^{-1}}\cdot a_l)\#v_{hg^{-1}})\delta_g\\
&=(a_{hgh^{-1}}\#v_{hg^{-1}})\delta_g, \
\end{align*}
because $v_{hgh^{-1}}.a_l\neq 0$ if and only if $hgh^{-1}=l$,
and $1_fa_l=a_l$ by Remark \ref{obs24}(v).

\vd

Since $h\in T_f\bigcap S_e$, it is routine to check that the map $g\mapsto hgh^{-1}$ from $G_e$ to $G_f$ is a bijection
and  induces the isomorphism of $K$-algebras $\theta_h:\bigoplus\limits_{g\in
G_e}A_g\to \bigoplus\limits_{g\in G_e}A_{hgh^{-1}}$ given by
$\theta_h(a_g)=a_{hgh^{-1}}$, for all $a_g\in A_g$.

\vd

Finally, the map $\gamma:\bigoplus\limits_{g\in
G_e}A_g\to\bigoplus\limits_{g\in G_e}(A_{hgh^{-1}}\#v_{hg^{-1}})\delta_g$ given by
$\gamma(a_g)=(a_{hgh^{-1}}\#v_{hg^{-1}})\delta_g$, for all $g\in G_e$ and $a_g\in A_g$,
is an isomorphism of $K$-algebras. Indeed, clearly $\gamma$ is an isomorphism of
$K$-modules (induced by $\theta_h$), and
\begin{align*}
\gamma(a_g)\gamma(b_l)&=(a_{hgh^{-1}}\#v_{hg^{-1}})\delta_g(b_{hlh^{-1}}\#v_{hl^{-1}})\delta_l\\
&=(a_{hgh^{-1}}\#v_{hg^{-1}})\beta_g(b_{hlh^{-1}}\#v_{hl^{-1}})\delta_{gl}\\
&=(a_{hgh^{-1}}\#v_{hg^{-1}})(b_{hlh^{-1}}\#v_{hl^{-1}g^{-1}})\delta_{gl}\\
&=(a_{hgh^{-1}}(v_{hg^{-1}(hl^{-1}g^{-1})^{-1}}\cdot b_{hlh^{-1}})\#v_{h(gl)^{-1}})\delta_{gl}\\
&=(a_{hgh^{-1}}(v_{hgh^{-1}}\cdot b_{hgh^{-1}})\#v_{h(gl)^{-1}})\delta_{gl}\\
&=(a_{hgh^{-1}}b_{hlh^{-1}}\#v_{h(gl)^{-1}})\delta_{gl}\\
&=(\theta(a_g)\theta(b_l)\#v_{h(gl)^{-1}})\delta_{gl}\\
&=(\theta(a_gb_l)\#v_{h(gl)^{-1}})\delta_{gl}\\
&=((a_gb_l)_{h(gl)h^{-1}}\#v_{h(gl)^{-1}})\delta_{gl}\\
&=\gamma(a_gb_l),
\end{align*}
for all $a_g\in A_g$ and $b_l\in A_l$. Therefore, we have from the above that
\begin{center}
$S_{e,f}=\bigoplus\limits_{g\in
G_e}(A_{hgh^{-1}}\#v_{hg^{-1}})\delta_g\simeq
\bigoplus\limits_{g\in G_e}A_g\,\,\,$\quad and\quad $\,\,\, C_{e,f}\simeq
M_{n_{e,f}}(\bigoplus\limits_{g\in G_e}A_g)$.
\end{center}
as unital $K$-algebras.
\end{proof}

\vd

\begin{teo}\label{teo45}

\[
C_0\simeq\bigoplus\limits_{e\in G_0}(\bigoplus\limits_{f\in
G_0}M_{n_{e,f}}(\bigoplus\limits_{g\in G_e}A_g)),
\]
as unital $K$-algebras.
\end{teo}

\begin{proof}
It follows by Lemmas \ref{lema41} and \ref{lema42}, and Proposition \ref{prop44}.
\end{proof}

\vu

\section{Final remarks}

Let $G$ be a finite groupoid, $R$ a unital $K$-algebra and $\beta=(\{E_g\}_{g\in G},
\{\beta_g\}_{g\in G})$ an action of $G$ on $R$ such that every $E_e$, $e\in
G_0$, is unital and $R=\bigoplus_{e\in G_0}E_e.$ By Proposition \ref{prop22}
$R$ is a $KG$-module algebra and we can consider the smash product algebra as in \cite{N},
which is constructed as follows:

\vu

Given a weak Hopf algebra $H$, denote by $H_t$ the target counital subalgebra of $H$ defined by
$H_t:=\{h\in H\mid\, \varepsilon_t(h)=h\}=\varepsilon_t(H)$, where $\varepsilon_t(h)=
\varepsilon(1_1h)1_2$, for every $h\in H$. The algebra $H$ has a natural structure of a
left $H_t$-module via multiplication, and any $H$-module algebra $X$ is a right $H_t$-module via
the antipode $S$ of $H$, that is, $x\cdot z:=S(x)\cdot x$, for all $x\in X$ and $z\in H_t$.
Hence, we can take the $K$-module $X\otimes_{H_t} H$, which has an structure of a
unital $K$-algebra induced by the multiplication $(x\otimes h)(y\otimes l)=
x(h_1\cdot y)\otimes h_2l$, for all $x,y\in X$ and $h,l\in H$. Its identity element is
$1_X\otimes 1_H$. Notice that $X\otimes_{H_t} H$ also is a left $H^*$-module algebra via
$h^*\cdot (x\otimes l)=x\otimes (h^*\rightharpoonup l)$, for all $h^*\in H^*$, $l\in H$ and $x\in X$,
so we can also consider the $K$-algebra $(X\bigotimes_{H_t} H)\bigotimes_{H_t^*} H^*$.

\vu

Our intent in this section is to present a natural exact sequence of $K$-algebras relating
$B=(R\star_\beta G)\# KG^*$, as considered in the subsection 2.4, to
$A=(R\otimes_{H_t} H)\otimes_{H_t^*} H^*$, in the case that $H=KG$.
Observe that $KG$ (resp., $KG^*$) is a weak Hopf algebra,
with antipode given by $S(u_g)=u_{g^{-1}}$ (resp., $S(v_g)=v_{g^{-1}}$). We start with
the following proposition. Recall that $T_e=\{g\in G\ |\ r(g)=e\}$, for all $e\in G_0$.

\begin{prop}\label{prop51}\,\, There exist a unital subalgebra $C$ of $B$ containing $B_0$
as subalgebra, and an ideal $D$ of $B$ such that $B=C\bigoplus D$ and $BD=DB=0$.
\end{prop}

\vu

\begin{proof}
Notice that $$B=\bigoplus\limits_{g,h\in G}E_g\delta_g\# v_h= C\bigoplus
D, \text{ where}\,\,\, C=\bigoplus\limits_{d(g)=r(h)}E_g\delta_g\#v_h\,\,\,\text{ and}
\,\,\, D=\bigoplus\limits_{d(g)\neq r(h)}E_g\delta_g\#v_h.$$ Furthermore,
$B_0=\bigoplus\limits_{r(g)=d(g)=r(h)}E_g\delta_g\# v_h$ is a direct summand of $C$.
It is a routine calculation to check that $C$ (resp., $B_0$) is a unital subalgebra of $B$ (resp., $C$)
with identity element $\sum\limits_{e\in G_0}1_e\delta_e\#\sum\limits_{g\in T_e} v_g$
(see Lemma \ref{lema32}(i)), as well as $D$ is an ideal of $B$.
We saw in the subsection 2.5 that $BD=0$. It follows by similar arguments that also $DB=0$.
\end{proof}

\vu

\begin{teo}\label{teo52}
The natural map $\varphi:B\to A$, given by $a_g\delta_g\# v_h\mapsto a_g\otimes u_g\otimes v_h$, induces
the following exact sequence of $K$-algebras
$$0\longrightarrow D\longrightarrow B\longrightarrow\varphi(C)\longrightarrow 0.$$
In particular, $B_0$ is isomorphic to a subalgebra of $A$.
\end{teo}

To prove this theorem we need first to describe the elements of $KG_t$ and $KG^*_t$.

\vd

\begin{lema}\label{lema53}
$$KG_t=\bigoplus\limits_{e\in G_0}Ku_e\quad\text{ and}\quad KG^*_t=\sum\limits_{e\in
G_0}K(\sum\limits_{h\in T_e}v_h).$$
\end{lema}

\vu

\begin{proof}

An element $x=\sum\limits_{g\in G}\lambda_gu_g\in KG$ satisfies $\varepsilon_t(x)=x$
if and only if
\begin{align*}
\sum\limits_{g\in G}\lambda_gu_g&=
\sum\limits_{g\in G}\lambda_g\varepsilon_t(u_g)=
\sum\limits_{g\in G}\lambda_g\sum\limits_{e\in G_0}\varepsilon(u_eu_g)u_e\\
&=\sum\limits_{g\in G}\lambda_g\varepsilon(u_{r(g)g})u_{r(g)}
=\sum\limits_{g\in G}\lambda_gu_{r(g)},
\end{align*}
if and only if $\lambda_g=0$, para all $g\notin G_0$.

\vd

An element $x=\sum\limits_{g\in G}\lambda_gv_g\in KG^*$ satisfies $\varepsilon_t(x)=x$
if and only if
\begin{align*}
\sum\limits_{g\in G}\lambda_gv_g&=\sum\limits_{g\in G}\lambda_g\varepsilon_t(v_g)\\
&=\sum\limits_{g\in G}\lambda_g(\sum\limits_{h\in G}\sum\limits_{{l\in G}\atop
{d(l)=d(h)}}\varepsilon(v_{hl^{-1}}v_g)v_l)\\
&=\sum\limits_{g\in
G}\lambda_g(\sum\limits_{{h\in G}\atop{r(h)=r(g)}}\varepsilon(v_g)v_{g^{-1}h})\\
&=\sum\limits_{g\in G_0}\lambda_g(\sum\limits_{h\in
T_g}v_h)
\end{align*}
if and only if $\lambda_h=\lambda_g$, for all $g\in G_0$ and $h\in T_g$.
\end{proof}

\vt

\nd{\bf Proof of Theorem 5.2:}

\vd

It is straightforward to check that the map $\varphi:B\to A$, given by $\varphi(a_g\delta_g\# v_h)=
a_g\otimes u_g\otimes v_h$, is a well defined homomorphism of $K$-algebras. Furthermore, the preunit
$1_{R\star_{\beta}G}\#1_{KG^*}$ of $B$ is taken by $\varphi$ onto the identity element $1_R\otimes 1_{KG}\otimes 1_{KG^*}$
of $A$. Indeed,
\begin{center}
$\varphi(1_{R\star_{\beta}G}\#1_{KG^*})=\varphi(\sum\limits_{e\in
G_0}1_e\delta_e\# \sum\limits_{g\in G}v_g)=\sum\limits_{e\in
G_0}\sum\limits_{g\in G}1_e\#u_e\#v_g$.
\end{center}

And, on the other hand, since $KG_t=\bigoplus\limits_{e\in G_0}Ku_e$ (Lemma \ref{lema53}), we have
\begin{align*}
1_R\#1_{KG}\#1_{KG^*}&=\sum\limits_{e\in
G_0}1_e\#\sum\limits_{f\in G_0}u_f\#\sum\limits_{g\in G}v_g =
\sum\limits_{e,f\in G_0}1_e.u_f\#u_f\#\sum\limits_{g\in G}v_g
\\
&=\sum\limits_{e,f\in G_0}S(u_f).1_e\#u_f\#\sum\limits_{g\in
G}v_g=\sum\limits_{e,f\in G_0}u_f.1_e\#u_f\#\sum\limits_{g\in
G}v_g\\
&=\sum\limits_{e,f\in G_0}\beta_f(1_e1_f)\#u_f\#\sum\limits_{g\in
G} v_g=\sum\limits_{e\in G_0}1_e\#u_e\#\sum\limits_{g\in G}v_g,
\end{align*}

This implies that the ideal $D$ of $B$ is contained in the kernel of $\varphi$ because
$(1_{R\star_{\beta}G}\#1_{KG^*})D=0$ (Proposition \ref{prop51}) and so
\begin{center}
$0=\varphi((1_{R\star_{\beta}G}\#1_{KG^*})D)=
\varphi(1_{R\star_{\beta}G}\#1_{KG^*})\varphi(D)=
1_A\varphi(D)=\varphi(D)$.
\end{center}
From this we also have
$\varphi(B)=\varphi(C)=\bigoplus\limits_{d(g)=r(h)}E_g\bigotimes_{KG_t} Ku_g\bigotimes_{KG^*_t}Kv_h$.

\vu

To end this proof it is enough to show that the $K$-algebras $\varphi(C)$
and $B/D$ are isomorphic. For this, take the map   $\psi:\bigoplus\limits_{d(g)=r(h)}E_g\times Ku_g\times
Kv_h\to B/D$, given by
$\psi(a_g,u_g,v_h)=\overline{a_g\delta_g\#v_h}$. Notice that for $x=\sum\limits_{e\in G_0}\lambda_eu_e\in KG_t$
we have
\begin{align*}
\psi(a_g.x,u_g,v_h)&=\psi(S(x).a_g,u_g,v_h)=\psi((\sum\limits_{e\in G_0}\lambda_eu_e).a_g,u_g,v_h)\\
&=\psi(\sum\limits_{e\in G_0}\lambda_e\beta_e(a_g1_e),u_g,v_h)\\
&=\psi(\lambda_{r(g)}a_g,u_g,v_h)=\overline{\lambda_{r(g)}a_g\delta_g\#v_h}
\end{align*}

\nd and
\begin{align*}
\psi(a_g,xu_g,v_h)&=\psi(a_g,\sum\limits_{e\in G_0}\lambda_eu_eu_g,v_h)=\psi(a_g,\lambda_{r(g)}u_g,v_h)=\overline{\lambda_{r(g)}a_g\delta_g\#v_h}.
\end{align*}
Hence, $\psi$ is $KG_t$-balanced.

\vd

Furthermore, for  $x=\sum\limits_{e\in
G_0}\lambda_e(\sum\limits_{l\in T_e}v_l)\in KG^*_t$ we also have
\begin{align*}
\psi(a_g,u_g.x,v_h)&=\psi(a_g,S(x).u_g,v_h)\\
&=\psi(a_g,S(\sum\limits_{e\in
G_0}\lambda_e(\sum\limits_{l\in T_e}v_l)).u_g,v_h)\\
&=\psi(a_g,(\sum\limits_{e\in
G_0}\lambda_e(\sum\limits_{l\in T_e}v_{l^{-1}})).u_g,v_h)\\
&=\psi(a_g,\sum\limits_{e\in
G_0}\lambda_e(\sum\limits_{l\in T_e}v_{l^{-1}}.u_g),v_h)\\
&=\psi(a_g,\lambda_{d(g)}u_g,v_h)\\
&=\psi(a_g,\lambda_{d(g)}u_g,v_h)\\
&=\overline{\lambda_{d(g)}a_g\delta_g\#v_h}
\end{align*}

\nd and

\begin{align*}
\psi(a_g,u_g,x.v_h)&=
\psi(a_g,u_g,(\sum\limits_{e\in
G_0}\lambda_e(\sum\limits_{l\in T_e}v_l))v_h)\\
&=\psi(a_g,u_g,\sum\limits_{e\in
G_0}\lambda_e(\sum\limits_{l\in T_e}v_lv_h))\\
&=\psi(a_g,u_g,\lambda_{r(h)}v_h)\\
&=\psi(a_g,u_g, \lambda_{r(h)})v_h\\
&=\overline{\lambda_{r(h)}a_g\delta_g\#v_h}\\
&=\overline{\lambda_{d(g)}a_g\delta_g\#v_h},
\end{align*}
Thus, $\psi$ also is $KG^*_t$-balanced. Therefore, $\psi$ induces a $K$-linear map
$\widetilde{\psi}$ from $\varphi(C)$ into $B/D$. It is immediate to see that this map is the inverse
of the $K$-algebra homomorphism $\widetilde{\varphi}$ from $B/D$ onto $\varphi(C)$ induced by $\varphi$.\qed

\vt

\bibliographystyle{amsalpha}




\vu


\end{document}